\documentclass[12pt]{article}
\usepackage{amsfonts,amsmath,amsthm, latexsym, amssymb, mathrsfs, graphicx}
\usepackage[dvipsnames]{xcolor}
\usepackage{bm, wasysym, tikz}
\usetikzlibrary{arrows, shapes}
\tikzset{>=stealth'}

\def\beq{\arraycolsep1pt\begin{eqnarray*}}
	\def\eeq{\end{eqnarray*}}

\newcommand{\sign}{{\rm{sign}}}

\begingroup
\newtheorem{theorem}{Theorem}[section]
\newtheorem*{theorem*}{Theorem}
\newtheorem{lemma}[theorem]{Lemma}
\newtheorem{proposition}[theorem]{Proposition}
\newtheorem{corollary}[theorem]{Corollary}
\newtheorem*{definition*}{Definition}
\newtheorem{remark}[theorem]{Remark}

\endgroup

\title{On the Lambert problem with drag}
\author{Antonio J. Ure\~na\\ \\
Departamento de Matem\'atica Aplicada, Facultad de Ciencias, \\
Universidad de Granada, 18071, Granada, Spain}

\date{}
\begin{document}

	\maketitle


	\begin{abstract}
		The Lambert problem consists in connecting two given points in a given lapse of time under the gravitational influence of a fixed center. While this problem is very classical, we are concerned here with situations where friction forces act alongside the Newtonian attraction. Under some boundedness assumptions on the friction, there exists exactly one rectilinear solution if the two points lie on the same ray, and at least two solutions travelling in opposite directions otherwise.
	\end{abstract}

\ \ \ {\em Keywords and phrases:} Kepler problem, Dirichlet boundary conditions, friction

\section{Introduction}

Finding a solution of the Kepler problem from two specified times and the corresponding respective positions is usually referred to as the {\em Lambert problem}. In other terms, it is the combination of the Kepler problem with Dirichlet boundary conditions.

\medbreak

 The history of this problem goes back in  time to the dawn of Celestial Mechanics, having been briefly mentioned by Lambert in a letter to Euler \cite{Bop}, and subsequently by Lagrange in his {\em Mécanique Analytique} \cite[\S 34, p. 39]{Lag}. 
 In Gauss' {\em Theoria Motus} \cite[\S 84, p. 108]{Gau} we read 
\begin{quote} 
  `Hence, inversely, it is apparent that  two radii vectors given in magnitude and position, together with the time in which the heavenly body describes the intermediate space, determine the whole orbit. But this problem, to be considered among the most important in the theory of the motions of the heavenly bodies, is not so easily solved, since the expression of the time in terms of the elements is transcendental, and moreover, very complicated.'
 \end{quote}
 \medbreak
 
 Fast forward to the second half of the twentieth century. In the sixties the development of computers and the needs of the aerospace industry gave rise to an important literature on numerical iterative algorithms designed to approximate solutions of the Lambert problem \cite{Eli, LanBla, Goo}. From a more theoretical point of view, the first results on existence and uniqueness are due to Sim\'o \cite{Sim}, whose approach was based on the Levi-Civita transformation. More recently, Albouy \cite[\S 38]{Alb}, \cite{Alb2}, has resorted to a related result, the so-called {\em Lambert theorem} to throw some new light on  Sim\'o's result. See also \cite{AlbUre1, AlbUre2}.

\medbreak

If the particle moves in the vacuum and is not affected by any forces other than the gravity of the fixed center, the problem is integrable, a fact already known to Newton. If on the other hand our particle crosses a cloud of gas or dust (or is so close to the Earth that it interacts with its atmosphere), then one should take into account the influence of the drag. Friction forces in Celestial Mechanics have also a tradition spanning for centuries. Their effect was already studied by Euler \cite{Eul} or Poincar\'e \cite[Chapter VI]{Poi}, but research in this direction continues to this day \cite{Sto,Cel,MarOrtReb, MarOrtReb2,Har}. See also the recent work \cite{Pan}, which studies numerically the Lambert problem in a frictional environment. 

\medbreak

Many forms of friction make sense from a physical point of view. See, e.g., the discussion in \cite[p. 266-267]{MarOrtReb2}.  In this paper we shall always assume that the friction force is linear in the velocity and acts in the opposite direction of motion. On the other hand, its intensity may depend in a complicated way on the position of the particle, especially if the environment is heterogeneous. Mathematically we are led to a system of differential equations of the form
\begin{equation*}
\label{eu1}\ddot x+D(x)\dot x=-\frac{x}{|x|^3}\,,\qquad x\in\mathbb R^2\backslash\{0\}\,,\tag{$K$}
\end{equation*}
where $D:\mathbb R^2\backslash\{0\}\to\mathbb R$ is nonnegative and continuously differentiable. Unless explicitly stated otherwise, solutions of \eqref{eu1} are understood in a classical sense, i.e., they are assumed to be twice continuously differentiable and avoid the collision singularity.

\medbreak

Given a solution $x=x(t)$ of \eqref{eu1}, its angular momentum $c(t):=\det(x(t),\dot x(t))$ satisfies the first-order linear equation $\dot c=-D(x(t))c$ and therefore, it cannot change sign. Passage to polar coordinates  $x=r(\cos\theta,\sin\theta)$ leads to the classical equality  $c=r^2\dot\theta$, and consequently the sign of $c$ divides the set of solutions of (\ref{eu1}) into three nonoverlapping classes: {\em solutions living in a ray} (or rectilinear), {\em solutions rotating counterclockwise, and solutions rotating clockwise}. 

\medbreak

Assuming that our solution is defined on the time interval $[-T,0]$, it will be called an {\em arc from $A:=x(-T)$ to $B:=x(0)$} provided that $|\theta(0)-\theta(-T)|<2\pi$. (For reasons that will be clear below we name the time interval $[-T,0]$ instead of the more conventional option $[0,T]$). 
On the other hand, the number $T$ is usually referred to as the {\em flight (or transfer) time}  of the arc. 

\medbreak

 We are now ready to formulate our problem in a more precise way. {\em Given points $A,B\in\mathbb R^2\backslash\{0\}$, and given some flight time $T>0$, are there arcs from $A$ to $B$ having flight time $T$?} With other words, we are concerned with the Dirichlet problem arising from the combination of \eqref{eu1} with the boundary conditions 
  \begin{equation}\label{bc}
 x(-T)=A,\qquad x(0)=B\,,\tag{$BC$}
 \end{equation}
 focusing our attention on solutions rotating for less than one tour on the given time interval.  Our precise assumptions on the (nonnegative, continuously differentiable) friction coefficient $D=D(x)$ will be as follows:
 \begin{itemize}
 	\item[{\bf [}{\bf D$_1$]}]\qquad $D:\mathbb R^2\backslash\{0\}\to\mathbb R$ is bounded. 
 	\item[{\bf [}{\bf D$_2$]}]\qquad $\lim_{x\to 0}\sqrt{|x|}\,\nabla D(x)=0$.
\end{itemize}
  The main result of this paper is given below:
 \begin{theorem}\label{th1}
 	{Assuming {\bf [D$_{1-2}$]}, fix some flight time $T>0$ and points $A,B\in\mathbb R^2\backslash\{0\}$. Then:
 		\begin{enumerate}
 			\item[(a)] If $A$ and $B$ lie on the same ray starting at the origin then there exists a unique rectilinear arc going from $A$ to $B$ in the flight time $T$.
 			\item[(b)] If $A,B$ do not lie on the same ray starting at the origin then there exists at least one arc from $A$ to $B$ with flight time $T$ and rotating clockwise, and at least one arc from $A$ to $B$ with flight time $T$ and rotating counterclockwise.
 		 		
 		 		\end{enumerate}} 
 \end{theorem}
Some remarks are in order:
\begin{enumerate}
\item[(i)] It seems reasonable to ask whether assumptions {\bf [D$_{1-2}$]} are actually necessary. While we cannot entirely answer to this question, it will be clear from our discussion (see Corollary \ref{cor42}) that {\bf [D$_{2}$]} may indeed be fully dropped in the case of the rectilinear statement {\em (a)}. Assumption {\bf [D$_{2}$]} will be required only in Sections \ref{sec6}-\ref{sec7} to complete the proof of {\em (b)}, and we do not know whether one could construct another proof without this hypothesis. It implies that the function $z\in\mathbb C\backslash\{0\}\mapsto D(z^2)$ can be extended to a continuously differentiable function of two real variables on $\mathbb C\equiv \mathbb R^2$,  and in particular, $D$ has a limit at $x=0$.
\item[(ii)] Concerning the nonrectilinear situation {\em (b)}, the two arcs rotating in opposite directions are well-known to be unique  in the frictionless case $D\equiv 0$ (see \cite{Sim}). We do not know whether uniqueness still holds under the presence of friction. 
\item[(iii)] Throughout this paper we focus our attention on motions  making less than one full tour on the given time interval, which we call {\em arcs} (they are sometimes called {\em simple arcs} in the literature).  It would be interesting to study the existence and multiplicity of solutions turning more than one tour; see, e.g. the recent paper \cite{AlbUre2} on the frictionless situation.

\item[(iv)]In the rectilinear case {\em (a)}, the unique solution is actually nondegenerate (see Lemma \ref{lem43}). Thus, when the two endpoints are slightly perturbated so as not lie on the same line, our solution can be continued in such a way that it sweeps a small angle (the so-called direct arc). These problems admit also solutions rotating in the opposite direction and therefore sweeping an angle close to $2\pi$ (indirec arcs); in the limit such solutions converge to a {\em generalized solution} which bounces at the origin. Bouncing solutions will not appear explicitly in this paper, but they will be somehow behind the arguments of Sections \ref{sec8}-\ref{sec7}.
\end{enumerate}	

After this introduction the paper is organized as follows. We begin with Sections \ref{sec31} and \ref{sec44}, where we discuss some general properties of the damped Kepler equation {\em (K)}. In Section \ref{sec22} we state, without proof, three important results labelled as Propositions \ref{prop11}, \ref{prop2} and \ref{prop12}, which will promptly lead to the proof of Theorem \ref{th1} in Section \ref{sec555}. The second part of the paper is devoted to prove the three propositions advanced in Section \ref{sec22}; more precisely Proposition  \ref{prop11} is proved in Section \ref{sec5}, Proposition \ref{prop2} is established in Section \ref{sec3}, and the remaining Sections \ref{sec8}-\ref{sec7} are devoted to validate Proposition \ref{prop12}. 

 \section{Catastrophes are due to collisions}\label{sec31}
 Equation \eqref{eu1} admits several alternative rewritings which can be used to reveal a number of its features. To start with, let the nonvanishing function $x=x(t)$, $t\in]\alpha,0]$, be continuously differentiable (our solutions will always be defined on time intervals ending at $t=0$ unless explicitly stated otherwise). We set
\begin{equation}\label{pp}
p(t):=\exp\left(-\int_{t}^0 D(x(s))ds\right).
\end{equation}
Then \eqref{eu1} becomes
$$\frac{d}{dt}\big(p(t)\dot x\big)=-\frac{p(t)}{|x|^3}x,\qquad x\not=0,$$
as one can readily check. While this equality is reminiscent of the usual presentation of Sturm-Liouville systems, the function $p$ appearing here depends on $x$ in a nonlinear, nonlocal fashion. Notice that
\begin{equation}\label{sti}
e^{D_*t}\leq p(t)\leq 1\,,\qquad t\in]\alpha,0]\,,
\end{equation}
where $D_*\geq 0$ stands for some upper bound of $D$ on $\mathbb R^2\backslash\{0\}$. If $x=x(t)$ is a solution of \eqref{eu1}, then it will satisfy
\begin{equation}\label{eu25}
\big|p(t)\dot x(t)-\dot x(0)\big|=\left|\int_{0}^{t}\frac{p(s)}{|x(s)|^3}x(s)ds\right|\leq\int_{t}^{0} \frac{1}{|x(s)|^2}ds\leq\frac{|t|}{\min_{t\leq s\leq 0}|x(s)|^2}\,,\qquad t\in]\alpha,0].
\end{equation}

We shall often work with solutions $x:]\alpha,0]\to\mathbb R^2\backslash\{0\}$ of \eqref{eu1} which are {\em maximal in the past}. It means that if the extended solution $\hat x:]\hat\alpha,0]\to\mathbb R^2\backslash\{0\}$ satisfies that $\hat\alpha\leq\alpha$ and $\hat x(t)=x(t)$ for every $t\in]\alpha,0]$, then $\hat\alpha=\alpha$. In this situation, if $\alpha>-\infty$ we may say that $x$ is not globally defined in the past, and solutions of this kind are the target of the following 

\begin{lemma}\label{lem31}{Let $x:]\alpha,0]\to\mathbb R^2\backslash\{0\}$ be a solution of \eqref{eu1}, maximal in the past. If 	$\alpha>-\infty$ then $\liminf_{t\downarrow\alpha}|x(t)|=0$.}
	\begin{proof}
The combination of \eqref{sti}-\eqref{eu25} gives
	$$|\dot x(t)|\leq e^{D_*|t|}|\dot x(0)|+ \frac{|t|\,e^{D_*|t|}}{\min_{t\leq s\leq 0}|x(s)|^2}\,,\qquad \alpha<t<0.$$ 	
		
		If we assume, using a contradiction argument, that $\liminf_{t\downarrow\alpha}|x(t)|>0$, then this inequality implies that $\limsup_{t\downarrow\alpha}|\dot x(t)|<+\infty$, and \eqref{eu1} gives $\limsup_{t\downarrow\alpha}|\ddot x(t)|<+\infty$. Therefore both $x(t)$ and $\dot x(t)$ have limits when $t\downarrow\alpha$, the limit of $x$ being nonzero. The standard continuation theory for solutions of ordinary differential equations states that our solution can be extended to some time interval containing $\alpha$. This is a contradiction and concludes the proof.  	\end{proof}
\end{lemma}

Assume now that $T>0$ is fixed and $x_n:[-T,0]\to\mathbb R^2\backslash\{0\}$ is a sequence of solutions of \eqref{eu1} satisfying 
$$x_n(0)\to x_0\not=0,\qquad \dot x_n(0)\to\dot x_0.$$

\begin{lemma}\label{lem322}
	{If the solution $x=x(t)$ of {(K)} with $x(0)=x_0$ and $\dot x(0)=\dot x_0$ cannot be extended to $[-T,0]$, then $\min_{[-T,0]}|x_n|\to 0$ as $n\to+\infty$.}
	
	\begin{proof}We argue by contradiction and assume that, after possibly passing to a subsequence, $\{\min_{[-T,0]}|x_n|\}_n$ is bounded from below by a positive constant. Combining \eqref{sti}-\eqref{eu25} with the fact that $\{\dot x_n(0)\}$ is bounded we conclude that $\{|\dot x_n|\}_n$ is uniformly bounded on $[-T,0]$. In addition $\{|x_n(0)|\}$ is bounded, and it follows that also $\{|x_n(t)|\}_n$ is uniformly bounded on $[-T,0]$. Moreover, both sequences $\{x_n(t)\}_n$, $\{\dot x_n(t)\}_n$ are equicontinuous (as a consequence of equation \eqref{eu1} in the latter case), and the Ascoli-Arzela lemma states that they are both uniformly convergent, at least along some subsequence. Then, $x(t):=\lim_{n\to+\infty}x_n(t)$ must be a solution of \eqref{eu1} satisfying $x(0)=x_0$, $\dot x(0)=\dot x_0$ and defined on $[-T,0]$. It contradicts our assumptions and concludes the proof.		
\end{proof}
\end{lemma}

\section{Rectilinear solutions of the Kepler equation}\label{sec44}
We devote this section to explore several features of the collinear motions of the Kepler problem. More precisely, let $w_0\in\mathbb R^2$ with $|w_0|=1$ be fixed and consider motions $x:]\alpha,0]\to\mathbb R^2\backslash\{0\}$ of the form $x(t):=r(t)w_0$ with $r=r(t)>0$. Setting $\delta(r):=D(r w_0)$, equation \eqref{eu1} becomes
\begin{equation}\label{1d}
\ddot r+\delta(r)\dot r=-\frac{1}{r^2},\qquad r>0.
\end{equation} 

Throughout this section we shall study this scalar equation under the assumption (modelled on {\bf [D$_{1}$]}) that $\delta:]0,+\infty[\to\mathbb R$ is {\em nonnegative, bounded and continuously differentiable} (growth assumptions on $\delta'$ near the origin are not required at this stage). We shall start with the following observation:
 \begin{lemma}\label{lem400}
 {Two solutions $r_1\not\equiv r_2$ of \eqref{1d} intersect at most once. With other words, if $r_1(t_*)=r_2(t_*)$ for some $t_*$ then $r_1(t)\not=r_2(t)$ for any $t\not=t_*$ in the common definition interval of $r_1$ and $r_2$.}	
\begin{proof}We use a contradiction argument and assume instead that there are solutions $r_1\not\equiv r_2$ of \eqref{1d} and times  $t_A<t_B$ in the common definition interval of $r_1$ and $r_2$ such that
	\begin{equation}\label{eqq25}
		r_1(t_A)=r_2(t_A)=:r_A,\qquad r_1(t_B)=r_2(t_B)=:r_B,\qquad   r_1(t)<r_2(t)\ \text{for every }t\in]t_A,t_B[.
	\end{equation}
	Let $\Delta:]0,+\infty[\to\mathbb R$ denote a primitive of $\delta$. Integration in both sides of \eqref{1d} leads to the equalities
	$$\dot r_i(t_B)-\dot r_i(t_A)=-\Delta(r_B)+\Delta(r_A)-\int_{t_A}^{t_B}\frac{1}{r_i(t)^2}dt,\qquad i=1,2,$$
	implying that
	$$\dot r_2(t_B)-\dot r_2(t_A)>\dot r_1(t_B)-\dot r_1(t_A),$$
	which is not possible since $\dot r_1(t_A)\leq\dot r_2(t_A)$ and  $\dot r_1(t_B)\geq\dot r_2(t_B)$, by \eqref{eqq25}. It concludes the proof.
	\end{proof}
 \end{lemma}

Let $r=r(t)$, $t\in]\alpha,0]$, be a solution of \eqref{1d}. Recalling the arguments at the beginning of Section \ref{sec31}, and setting 
\begin{equation}\label{pn31}
p(t):=\exp\left(-\int_t^0\delta(r(s))ds\right),
\end{equation}
we see that 
\begin{equation}\label{1d+}
\frac{d}{dt}(p(t)\dot r)=-\frac{p(t)}{r^2},\qquad r>0.
\end{equation}
On the other hand, denoting by $D_*\geq 0$ an upper bound of $\delta$ on $]0,+\infty[$ one checks that the inequalities \eqref{sti} still hold in this situation.

\medbreak

 Fix numbers $T,r_B>0$ and consider the set $\mathcal I$ of final speeds $v\in\mathbb R$ such that the solution $r=r(t)$ of \eqref{1d} with $r(0)=r_B$ and $\dot r(0)=v$ is defined in the past up to time $t=-T$. The usual smooth dependence theorems state that $\mathcal I$ is open and the function $$\mathfrak R:\mathcal I\to\mathbb R,\qquad v\mapsto r(-T),$$ 
is continuously differentiable on $\mathcal I$. The main result of this section collects some basic properties of $\mathcal I$ and $\mathfrak R$:
\begin{lemma}\label{lem41}
	{The following hold:
		\begin{enumerate}
		\item[(i)] There exists some $\beta\in\mathbb R$ such that $\mathcal I=]-\infty,\beta[$.
			\item[(ii)] $\mathfrak R$ establishes a decreasing diffeomorphism from $]-\infty,\beta[$ into $]0,+\infty[$. With formulas,
		\begin{equation*}
			\mathfrak R'(v)<0\text{ for every } v<\beta\,,\qquad\lim_{v\to-\infty}\mathfrak R(v)=+\infty,\qquad \lim_{v\to\beta_-}\mathfrak R(v)=0.
		\end{equation*}
		
	\end{enumerate}}
\begin{proof}
The fact that $\mathcal I$ is open and $\mathfrak R:\mathcal I\to\mathbb R$ is continuously differentiable is a direct consequence of the usual theorems of smooth dependence on initial conditions. Moreover, it follows from Lemmas \ref{lem31} and \ref{lem400} that the set $\mathcal I$ is an interval and $\mathfrak R:\mathcal I\to\mathbb R$ is strictly decreasing.

\medbreak

We claim first that $\mathcal I$ is unbounded from below and $\lim_{v\to-\infty}\mathfrak R(v)=+\infty$. It can be done by picking some sequence $\{r_n\}_n$ of solutions of \eqref{1d} with 
$$r_n(0)=r_B\text{ for every }n,\qquad 0>v_n:=\dot r_n(0)\to-\infty.$$ Each function $r_n$ is defined on some interval $]\alpha_n,0]$, maximal to the left. If $\dot r_n(t)<0\ \forall t\in]\max(\alpha_n,-T),0[$ then Lemma \ref{lem31} implies that $\alpha_n<-T$ and we set $a_n:=-T$. Otherwise, $\dot r_n(t)=0$ for some $t\in]\max(\alpha_n,-T),0[$, and we denote by $a_n$ the maximum of such numbers $t$. In any case, $r_n(t)\geq r_B$ for all $t\in[a_n,0]$, and defining $p_n:[a_n,0]\to\mathbb R$ as in \eqref{pn31} for $r=r_n$, equality \eqref{1d+} gives
\begin{equation*}
\left|\frac{d}{dt}\big(p_n(t)\dot r_n\big)\right|=\frac{p_n(t)}{r_n(t)^2}\leq\frac{e^{D_*T}}{r_B^2} ,\qquad a_n\leq t\leq 0,\qquad n\in\mathbb N,
\end{equation*}
from where it follows that $\max_{[a_n,0]}\dot r_n\to-\infty$ as $n\to+\infty$. In combination with Lemma \ref{lem31} it implies the claim. 

\medbreak

We observe next that $\mathcal I$ is bounded from above. Arguing by contradiction, we assume the existence of a second sequence $\{r_n\}_n$ of solutions of \eqref{1d} with $v_n:=\dot r_n(0)\to+\infty$, all of them defined on $[-T,0]$ and satisfying $r_n(0)=r_B$. It follows from \eqref{1d+} that 
$$\frac{d}{dt}(p_n(t)\dot r_n(t))< 0,\qquad t\in[-T,0],$$
where each $p_n:[-T,0]\to\mathbb R$ is defined as in \eqref{pn31} for $r=r_n$. It implies that $\dot r_n(t)\to+\infty$ uniformly with respect to $t\in[-T,0]$, which is not possible since all $r_n$ are positive.

\medbreak

We also need to check that $\lim_{v\to\beta_-}\mathfrak R(v)=0$. This statement follows from the combination of Lemma \ref{lem322} with the observation that solutions $r=r(t)$ of \eqref{1d} do not have local minima in open time intervals.

\medbreak

It remains to show that $\mathfrak R'(v)\not=0$ for every $v<\beta$. We use a contradiction argument and assume instead that $\mathfrak R'(v_*)=0$ for some $v_*<\beta$. It implies the existence of some solution $r_*=r_*(t)$ of \eqref{1d} such that the linear Dirichlet problem 
\begin{equation}\label{lin}
\ddot u+\delta'(r_*(t))\dot r_*(t)u+\delta(r(t))\dot u=\frac{2u}{r_*(t)^3},\qquad u(-T)=u(0)=0,
\end{equation}
has a nonzero solution $u:[-T,0]\to\mathbb R$. After possibly replacing $u$ by $-u$ and $T$ by some smaller time there is no loss of generality in further assuming that
$u(t)>0$ for every $t\in]-T,0[$. Noting that $\delta'(r_*)\dot r_*u+\delta(r_*)\dot u=(d/dt)(\delta(r_*)u)$, integration in \eqref{lin} gives
$$\dot u(0)-\dot u(-T)=2\int_{-T}^0\frac{u(t)}{r_*(t)^3}\,ds>0,$$
which is not possible since $\dot u(-T)<0<\dot u(0)$. This contradiction concludes the proof.
\end{proof}
\end{lemma}

One immediately arrives to the following reformulation of Theorem \ref{th1}{\em (a)}, where no traces of assumption {\bf [D$_2$]} are present:
\begin{corollary}\label{cor42}{Let $\delta:]0,+\infty[\to\mathbb R$ be continuously differentiable, nonnegative and bounded. Then for every $r_A,r_B>0$ and every $T>0$ there exists a unique solution of \eqref{1d} with $r(-T)=r_A$ and $r(0)=r_B$.}
\end{corollary}

We close this section by exploring the nondegeneracy of the rectilinear solutions of the Kepler equation \eqref{eu1}. Some one-dimensional nondegeneracy was already established in Lemma \ref{lem41}{\em (ii)}, but  we would like to show nondegeneracy in the context of the planar Dirichlet problem \eqref{eu1}-\eqref{bc}. More precisely, we shall adopt the following
\begin{definition*}{A solution $x_*:[-T,0]\to\mathbb R^2\backslash\{0\}$ of \eqref{eu1} will be called nondegenerate if the variational equation 
		\begin{equation}\label{linp}
			\ddot w+\langle\nabla D(x_*(t)),w\rangle\dot x_*(t)+D(x_*(t))\dot w=-\frac{1}{|x_*(t)|^3}\,w+3\frac{\langle x_*(t),w\rangle}{|x_*(t)|^5}\,x_*(t)\,,\qquad w\in\mathbb R^2\,,
		\end{equation} 
together with the homogeneous Dirichlet boundary conditions $w(-T)=0=w(0)$, admit only the trivial solution $w\equiv 0$.}
\end{definition*}
\begin{lemma}\label{lem43}{Every rectilinear solution $x_*:[-T,0]\to\mathbb R^2\backslash\{0\}$ of \eqref{eu1} is nondegenerate.}
\begin{proof}Set $x_*(t):=r_*(t)w_0$ where $w_0\in\mathbb R^2$ is unitary and $r_*(t)>0$ for every $t\in[-T,0]$. Equation \eqref{linp} becomes
\begin{equation}\label{linp4}
	\ddot w+\langle\nabla D(r_*(t)w_0),w\rangle\dot r_*(t)w_0+\delta(r_*(t))\dot w=-\frac{1}{r_*(t)^3}\,w+3\frac{\langle w_0,w\rangle}{r_*(t)^3}\,w_0\,,\qquad w\in\mathbb R^2\,,
\end{equation} 	
where $\delta(r):=D(rw_0),\ r>0$. Let $w:[-T,0]\to\mathbb R^2$ be a solution with $w(-T)=0=w(0)$; then $v(t):=\det(w(t),w_0)$ satisfies the linear second-order equation
$$\ddot v+\delta(r_*(t))\dot v+\frac{v}{r_*(t)^3}=0,\qquad t\in[-T,0],$$
which also admits the positive solution $r_*:[-T,0]\to\mathbb R$. We use the method of reduction of order and set $v=r_*(t)v_1$, to obtain
$$\ddot v_1=-\left(2\frac{\dot r_*(t)}{r_*(t)}+\delta(r_*(t))\right)\dot v_1\,,\qquad v_1(-T)=v_1(0)=0\,.$$
The boundary conditions imply the existence of some $t_0\in]-T,0[$ such that $\dot v_1(t_0)=0$, subsequently the differential equation implies that $\dot v_1\equiv0$, and again by the boundary conditions, $v_1\equiv 0$. Consequently, $\det(w(t),w_0)\equiv0$, implying the existence of some function $u:[-T,0]\to\mathbb R$ such that $w(t)=u(t)w_0$ for every $t\in[-T,0]$. Going back to \eqref{linp4} we see that $u$ must be a solution of \eqref{lin}, and we deduce that $\dot u(0)\,\mathfrak R'(r_*(0))=u(-T)=0$. But $\mathfrak R'(r_*(0))<0$, and so, $\dot u(0)=0$, so that $u\equiv0$ by uniqueness. Therefore $w\equiv 0$, thus concluding the proof.
\end{proof}
\end{lemma}

\section{Three cornerstones supporting the proof}\label{sec22}

The purpose of this section is to bring forward three important results, labelled as Propositions \ref{prop11}, \ref{prop2} and \ref{prop12}, which will hold up the proof of Theorem \ref{th1}{\em (b)} in Section \ref{sec555}. In order to keep the pace of the exposition their proofs will be postponed to Sections \ref{sec5}-\ref{sec7}, in the second part of the paper. 
\begin{proposition}\label{prop11}{Nonrectilinear solutions  of \eqref{eu1} are globally defined in the past. With other words, if the nonrectilinear solution $x:]\alpha,0]\to\mathbb R^2\backslash\{0\}$ is maximal in the past, then $\alpha=-\infty$. }
\end{proposition}
Proposition \ref{prop11} was previously proved in \cite[Proposition 2.1]{MarOrtReb} assuming that $D$ is constant, and in  \cite[Proposition 2.1]{MarOrtReb2} when $D=D(|x|)$ depends only on the height of the particle. Notice that this result does not need assumptions  {\bf [D$_{1-2}$]}. Nonrectilinear solutions are also globally defined in the future; however, we shall  not need this fact in our analysis. 

\medbreak

A second ingredient which shall be needed in the next section concerns the existence of a priori bounds for the final speed of solutions joining two given heights in a given flight time. More precisely, assume that $r_A,r_B>0$ (in addition to $T>0$) are fixed and consider the boundary conditions

\begin{equation}\label{eu12}
	|x(-T)|=r_A,\qquad|x(0)|=r_B.	
\end{equation} 

\begin{proposition}\label{prop2}
	{There exists some $M>0$ (depending on $r_A$, $r_B$ and $T$ but not on $x$), such that whenever $x:[-T,0]\to\mathbb R^2\backslash\{0\}$ is a solution of \eqref{eu1}-\eqref{eu12} then $|\dot x(0)|<M$.
	}
\end{proposition}
We do not know whether a similar property holds for the initial (in the place of final) speed, and this is the reason why we use time-backwards Poincar\'e maps instead of their more traditional time-forward cousins. Proposition \ref{prop2} ensures that no families of solutions of \eqref{eu1}-\eqref{eu12} blow up to infinity, thus  playing an important role in the application of continuation arguments of topological degree.

\medbreak

We go back now to Proposition \ref{prop11}. There is an alternative way to present this result by using the language of Poincar\'e maps, which, throughout this paper, will be referred to backward time. Thus, given $T>0$, the associated Poincar\'e map $\mathcal P=\mathcal P_T$ maps an initial condition $(x_0,\dot x_0)\in(\mathbb R^2\backslash\{0\})\times\mathbb R^2$  into the pair $(x(-T),\dot x(-T))\in\mathbb R^2\times\mathbb R^2\equiv\mathbb R^4$. Here $x=x(t)$ stands for the solution of \eqref{eu1} satisfying the  initial -or rather, {\em final\,}-  conditions
\begin{equation}\label{ic}
	x(0)=x_0,\qquad \dot x(0)=\dot x_0.\tag{{\em FC}}	
\end{equation}
The natural domain of $\mathcal P$ is the set $\Omega=\Omega_T$ of points $(x_0,\dot x_0)\in(\mathbb R^2\backslash\{0\})\times\mathbb R^2$ such that the solution of \eqref{eu1}-(\ref{ic}\,) can be extended to $t=-T$. The usual continuous dependence theorems state that $\Omega$ is open in $\mathbb R^4$ and $\mathcal P:\Omega\to\mathbb R^4$ is continuous. Proposition \ref{prop11} above can be reformulated by saying that 
\begin{equation}\label{Om}
	\Omega\supset\{(x_0,\dot x_0)\in\mathbb R^2\times\mathbb R^2:\ x_0,\dot x_0\text{ are linearly independent}\}. 	
\end{equation}

We denote by $\mathcal X:\Omega\to\mathbb R^2$ the first two (position) components of $\mathcal P$. With other words, $\mathcal X=\Pi\circ\mathcal P$, where $\Pi:\mathbb R^2\times\mathbb R^2\to\mathbb R^2$ stands for the projection $(x,\dot x)\mapsto x$. Our third  postulate will be the following:
\begin{proposition}\label{prop12}{$\mathcal X$ admits a continuous extension $\bar{\mathcal X}:(\mathbb R^2\backslash\{0\})\times\mathbb R^2\to\mathbb R^2$. Moreover, $\bar{\mathcal X}(x_0,\mathbb R\, x_0)\subset[0,+\infty[\,x_0$, for every $x_0\in\mathbb R^2\backslash\{0\}.$}
\end{proposition}

Thus, collinear solutions may collide  with the singularity, but admitting that they bounce back at the collision one obtains a flow that is continuous in its position components. This result will be obtained with the help of (Levi-Civita's) regularization theory. 

\medbreak

Before closing this section we point out a consequence of the (still unproved) Propositions \ref{prop11} and \ref{prop12} which will play an important role in the proof of Theorem \ref{th1}{\em (b)}. We go back to the situation described in Lemma \ref{lem322} and suppose that $x_n:[-T,0]\to\mathbb R^2\backslash\{0\}$ is a sequence of solutions of \eqref{eu1}. Therefore $(x_n(0),\dot x_n(0))\in\Omega_T$ for every $n$, and we further assume that $(x_n(0),\dot x_n(0))\to (x_0,\dot x_0)\in\partial\Omega_T$. For each $n\in\mathbb N$ we denote by $\Delta\theta_n$ the angle swept by $x_n$ on the time interval $[-T,0]$. With formulas,
$$\Delta\theta_n:=\theta_n(0)-\theta_n(-T),$$
where $\theta_n(t):=\arg(x_n(t))$ is a continuous choice of the argument on $x_n$. 

\begin{corollary}\label{lem32}
	{If there are constants $r_A,r_B>0$ such that $|x_n(-T)|=r_A$ and $|x_n(0)|=r_B$ for every $n\in\mathbb N$ then $\lim\inf_{n\to+\infty}|\Delta\theta_n|\geq 2\pi$.}
	
	\begin{proof}By assumption, $|x_n(0)|=r_B\ \forall n\in\mathbb N$, and we see that $|x_0|=r_B$. The combination of Proposition \ref{prop11} -in the form given in \eqref{Om}- and Proposition \ref{prop12}, implies that 
		$$\dot x_0\in\mathbb R x_0\,,\qquad x_n(-T)\to\frac{r_A}{r_B}x_0\text{ as }n\to+\infty.$$ 
		
		Consequently, either $\lim\inf_{n\to+\infty}|\Delta\theta_n|\geq 2\pi$ as claimed, or otherwise $\lim_{n\to+\infty}\Delta\theta_n= 0$. In the latter situation the functions $\xi_n(t):=\langle x_n(t),x_0\rangle/r_B$ are positive for $n$ big enough, and they satisfy 
		$$\xi_n(-T)\to r_A,\qquad\qquad \xi_n(0)\to r_B,\qquad \min_{-T\leq t\leq 0}\xi_n(t)\to 0,\qquad\qquad\qquad n\to+\infty,$$
		the last fact following from {\em (i)}. On the other hand, \eqref{eu1} implies that
		$$\ddot \xi_n+D(x_n)\dot \xi_n=-\frac{\xi_n}{|x_n(t)|^3}<0,$$
		which, when evaluated at the point where $\xi_n$ attains its minimum leads to contradiction. It concludes the reasoning.
	\end{proof}
\end{corollary}

With the aim of speeding up our main arguments we put off the proofs of Propositions \ref{prop11}, \ref{prop2} and \ref{prop12} to Sections \ref{sec5}-\ref{sec7}. Nevertheless we shall rely on them on Section \ref{sec555} to prove Theorem \ref{th1}{\em (b)}.

\section{Brouwer degree and continuation from the rectilinear problem}\label{sec555}

In this section we shall prove Theorem \ref{th1} by relying on Propositions \ref{prop11}, \ref{prop12} and \ref{prop2} and using some basic techniques from Brouwer degree theory.  We refer the reader to the classical textbook \cite{Llo} or the recent treatise \cite{DinMaw} for comprehensive introductions to Brouwer degree.

\medskip

An important property of the Brouwer degree is its invariance by homotopies. If two problems are connected by a homotopy in such a way that solutions do not escape through the boundary, then they have the same degree. In particular, if the first problem has nonzero degree, then the second has a solution.
\medskip

In this section we shall need a particular consequence of this property that we describe next. Let $\mathcal U\subset\mathbb R^N$ be open (but not necessarily bounded), and let $$\Phi:[0,1]\times\mathcal U\to\mathbb R^N,\qquad (\lambda,u)\mapsto\Phi(\lambda,u),$$ be continuous and admit a continuously-defined derivative with respect to the $u$ variables, denoted  $\Phi'_u:[0,1]\times\mathcal U\to\mathbb R^{N\times N}$. We set $\Sigma:=\{(\lambda,u)\in[0,1]\times\mathcal U:\Phi(\lambda,u)=0\}$ and assume that 
\begin{enumerate}
	\item[{\em (a)}] There exists some $u_0\in\mathcal U$ such that $\Sigma\cap(\{0\}\times\mathbb R^N)=\{(0,u_0)\}$ and $\det(\Phi_u'(0,u_0))\not=0$.
	\item[{\em (b)}] There exists some constant $M>0$ such that $|u|<M$ for every $(\lambda,u)\in\Sigma$.
	\item[({\em c})] $\overline\Sigma\cap([0,1]\times(\partial{\mathcal U}))=\emptyset$.
\end{enumerate} 
\begin{lemma}\label{lem51}{Under the above, for each $\lambda\in[0,1]$ there exists some $u_\lambda\in\mathcal U$ such that $(\lambda,u_\lambda)\in\Sigma$. }
	\begin{proof}
(See Fig. 1(a)). Assumptions {\em (b)-(c)} guarantee that there exists some $\rho>0$ such that $u+\bar B(\rho)\subset\mathcal U$ for any $(\lambda,u)\in\Sigma$. Here, $\bar B(\rho)$ denotes the closed ball of radius $\rho$ in $\mathbb R^N$. Set
		$${\mathcal V}:=\{u\in\mathbb R^N\text{ such that }|u|<M\text{ and }u+\bar B(\rho)\subset\mathcal U\},$$
		which is open and bounded. Moreover, $\bar{\mathcal V}\subset\mathcal U$ and $\Sigma\subset[0,1]\times\mathcal V$. Therefore, the Brouwer degree  $\deg_B(\Phi(\lambda,\cdot),\mathcal V)$ does not depend on $\lambda\in[0,1]$. For $\lambda=0$, $\deg_B(\Phi(0,\cdot),\mathcal V)=\sign(\det\Phi_u'(0,u_0))=\pm 1$, and thus, 
		\begin{equation*}
		\deg_B(\Phi(\lambda,\cdot),\mathcal V)\not=0\text{ for every }\lambda\in[0,1].
		\end{equation*}
		 The result follows. 	\end{proof}
	
\end{lemma}
\begin{remark}
	{Well-known arguments going back to Leray-Schauder (\cite[Théorème Fondamental, p. 63]{LerSch}) show, under the assumptions above, the existence of a connected set of solutions sweeping all values of $\lambda$. This fact will not be used in our argument.}
\end{remark}

\begin{figure}[h!]
	\begin{center}
		\includegraphics[scale=1,trim={0.5cm 0cm 0cm 0cm},clip]{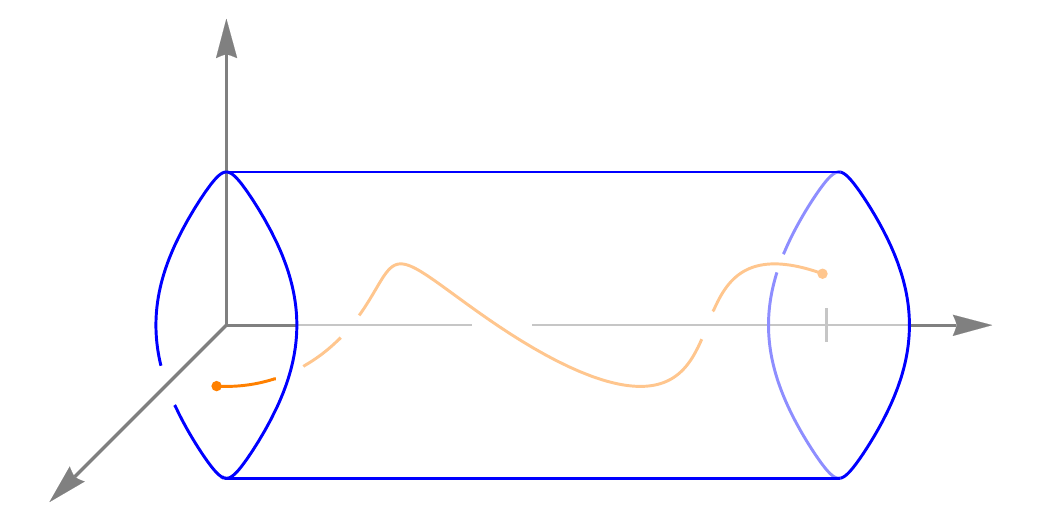}\put(-20,40){\footnotesize $\lambda\rightarrow$}\put(-64,40){\footnotesize $1$}\put(-180,80){\small $\Sigma$}\put(-140,-20){\small (a)}
		\hspace{1.6cm}	\includegraphics[scale=1]{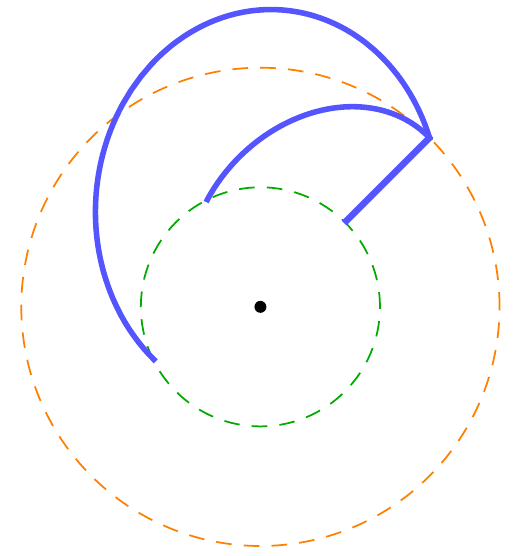}\put(-80,-20){\small (b)}\put(-15,120){\footnotesize $B$}\put(-63,86){\footnotesize $A_0$}\put(-97,91){\footnotesize $A_\lambda$}\put(-102,53){\footnotesize $A$}\put(-97,91){\footnotesize $A_\lambda$}
	\end{center}
\caption{(a): The projection of $\Sigma$ on the $\lambda$-variable is the full interval $[0,1]$. (b): Rotating the starting point on a circumference centered at the origin we obtain a homotopy from a rectilinear problem to the original one.}
\end{figure}
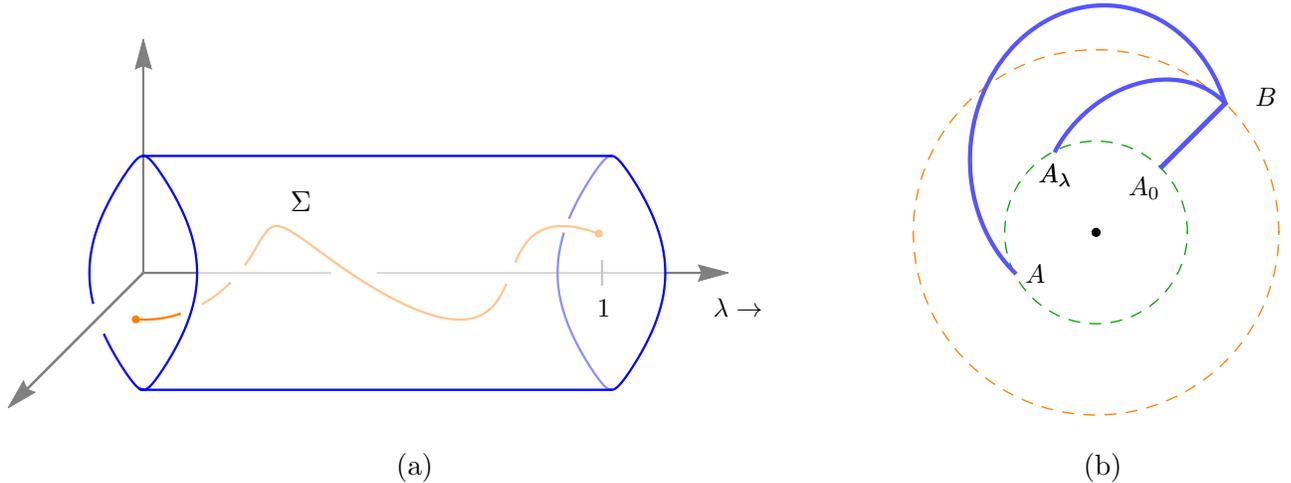

\begin{proof}[Proof of Theorem \ref{th1}(b)]
Write $A=r_A(\cos\theta_A,\sin\theta_A)$ and $B=r_B(\cos\theta_B,\sin\theta_B)=r_B w_0$, where $r_A,r_B>0$, $|\theta_B-\theta_A|<2\pi$, and $w_0:=(\cos\theta_B,\sin\theta_B)$. We consider the map $\Psi$ sending each final speed $v_0\in\mathbb R^2$ into the polar coordinates $(r(-T),\theta(-T))$ of the position $x(-T)=\mathcal X(B,v_0)$, of the solution $x$ of \eqref{eu1} satisfying $x(0)=B$, $\dot x(0)=v_0$. The lifting of the angle $\theta=\theta(t)$ is chosen so that $\theta(0)=\theta_B$.  In view of Proposition \ref{prop11} and Lemma \ref{lem41}{\em (i)}, the map $\Psi$ is naturally defined on the open set
$$\mathcal U:=\mathbb R^2\backslash([\beta,+\infty[w_0),$$ for some $\beta\in\mathbb R$. We also set 
\begin{equation*}
\Phi:[0,1]\times\mathcal U\to\mathbb R^2,\qquad\Phi(\lambda,v_0):=\Psi(v_0)-\big(
r_A,\vartheta(\lambda)\big),
\end{equation*}
where $\vartheta(\lambda):=(1-\lambda)\theta_B+\lambda\theta_A$. Denoting $A_\lambda:=r_A(\cos\vartheta(\lambda),\sin\vartheta(\lambda))$ we see that, for $\lambda=0$, $A_0=(r_A/r_B)B$ lies on the same ray as $B$, and thus $\Phi(0,v_0)=0$ if and only if $v_0$ is the final velocity of a rectilinear arc from $A_0$ to $B$ in the flight time $T$. In the case $\lambda=1$ it is clear that $\Phi(1,v_0)=0$ if and only if $v_0$ is the final velocity of an arc from $A_1=A$ to $B$ in the flight time $T$. With an additional remark: the direction of rotation must be set according to the sign of $\theta_B-\theta_A$, clockwise for $\theta_B-\theta_A<0$ and counterclockwise for $\theta_B-\theta_A>0$. Finally, for $0<\lambda<1$ one checks that $\Phi(\lambda,v_0)=0$ if and only if $v_0$ is the final velocity of an arc from $A_\lambda$ to $B$ in the flight time $T$ (the previous comments on the direction of the rotation still apply). See Fig. 1(b).

\medbreak

Theorem \ref{th1}{\em (b)} follows by applying Lemma \ref{lem51} to this function.  In order to check assumption {\em (a)} we observe that if $v_0\in\mathcal U$ satisfies that $\Phi(0,v_0)=0$ then $v_0=\mu w_0$ for some $\mu\in\mathbb R$ and our solution is rectilinear. Corollary \ref{cor42} then implies that there exists a unique solution of this type, which is nondegenerate by Lemma \ref{lem43}.  Assumption {\em (b)} is ensured by Proposition \ref{prop2}, while {\em (c)} is actually a consequence of  Corollary \ref{lem32}. The proof is complete.
\end{proof}

\section{Nonrectilinear solutions are defined for all (past) time}\label{sec5}
In this section we assume {\bf [D$_{1}$]}. Let us go to polar coordinates and replace the cartesian dependent variables $x\in\mathbb R^2\backslash\{0\}$ by $r:=|x|>0$ and $\theta:=\arg(x)\in\mathbb R/2\pi\mathbb Z$. The new equations are simpler if one introduces the angular momentum $c:=\det(x,\dot x)=r^2\dot\theta$, which is a function of time. System \eqref{eu1} yields
\begin{equation*}
\ddot r+D(x)\dot r+\frac{1}{r^2}-\frac{c^2}{r^3}=0,\qquad\dot c+D(x)c=0,
\end{equation*}
or, equivalently (assuming that $x=x(t)$ is defined on the time interval $]\alpha,0]$ for some $\alpha<0$),
\begin{equation}\label{eq101}
\frac{d}{dt}(p\dot r)=p\left(\frac{c^2}{r^3}-\frac{1}{r^2}\right),\qquad c(t)=\frac{c(0)}{p(t)},
\end{equation}
the function $p=p(t)>0$ being given by \eqref{pp}.
We shall also introduce the {\em potential} and the {\em total} energies 
\begin{equation}\label{euu}
v:=-\frac{1}{r}+\frac{c^2}{2r^2},\qquad h:=\frac{\dot r^2}{2}+v,	
\end{equation}
which are again functions of $t$. A direct computation shows that  $\dot h=-D(x)(\dot r^2+c^2/r^2)\leq 0$, and thus, $h$ is decreasing. Through Sections \ref{sec5}-\ref{sec3}, {\em we shall think of the six functions $r,\theta,c,p,v,h$ as associated to the solution $x$.} 

\medbreak

\begin{proof}[Proof of Proposition \ref{prop11}] We follow along the lines of \cite[Proposition 2.1]{MarOrtReb}, which deals with the special case $D\equiv{\rm const}$. See also \cite[Proposition 2.1]{MarOrtReb2}, where solutions are shown to be globally-defined in the future under the assumption that $D=D(|x|)$ depends only on the height of the particle.

		\medbreak
	
		We shall argue by contradiction and assume that $\alpha>-\infty$. Lemma \ref{lem31} then states that $\liminf_{t\downarrow\alpha}r(t)=0$. In view of \eqref{eq101} one has the inequality
			\begin{equation}\label{eu11}
			\frac{d}{dt}\Big(p(t)\dot r(t)\Big)\geq\\ p(t)\left(\frac{c(0)^2}{r(t)^3}-\frac{1}{r(t)^2}\right),\qquad t\in]\alpha,0].		
			\end{equation}
		The function $x\mapsto c(0)^2/x^3-1/x^2$ is bounded from below on $]0,+\infty[$. Together with \eqref{sti} we see that there exists some constant $M>1$ such that
			\begin{equation}\label{eu9}
			\frac{d}{dt}\Big(p(t)\dot r(t)\Big)\geq -M,\qquad t\in]\alpha,0].
			\end{equation}
			
			On the other hand, $r$ is positive and $\liminf_{t\downarrow\alpha}r(t)=0$, and we deduce that 
			\begin{equation}\label{eq11}
			\limsup_{t\downarrow\alpha}\big(p(t)\dot r(t)\big)\geq 0.
			\end{equation}
			Therefore, \eqref{eu9} it implies  that
			$$p(t)\dot r(t)\geq -M(t-\alpha),\qquad t\in]\alpha,0],$$
			and consequently (by \eqref{sti}),
			$$\dot r(t)\geq -\frac{M(t-\alpha)}{p(t)}\geq M\alpha e^{-D_*\alpha},\qquad t\in]\alpha,0].$$
			We have shown that $\dot r$ is bounded from below on $]\alpha,0]$. It implies that the limit $\lim_{t\downarrow\alpha}r(t)$ exists, and since  $\liminf_{t\downarrow\alpha}r(t)=0$ it follows that $\lim_{t\downarrow\alpha}r(t)=0$. Therefore, \eqref{eu11} gives
			$$\lim_{t\downarrow\alpha}\frac{d}{dt}\big(p(t)\dot r(t)\big)=+\infty,$$
			which, in combination with \eqref{eq11} implies that $\dot r(t)>0$ for $t$ sufficiently close to $\alpha$. Let now $\hat v,\hat h:]\alpha,0]\to\mathbb R$ be defined by
			$$\hat v(t):=\frac{c(0)^2}{2r(t)^2}-\frac{1}{r(t)},\qquad\hat h(t):=\frac{1}{2}\dot r(t)^2+\hat v(t),$$ and observe that
			\begin{equation*}
			\lim_{t\downarrow\alpha}\hat v(t)=\lim_{t\downarrow\alpha}\hat h(t)= +\infty.
			\end{equation*}
			In particular, $\hat v(t)>0$ for $t$ close to $\alpha$, and one has
			$$\frac{d\hat h}{dt}=\dot r(t)\left[-D(x(t))\dot r(t)+\frac{c(t)^2-c(1)^2}{r(t)^3}\right]\geq -D(x(t))\dot r(t)^2\geq -2D_*\hat h(t),$$
			which, in combination with Gronwall's Lemma contradicts the previous observation that $\hat h(t)\to+\infty$ as $t\downarrow\alpha$. This contradiction concludes the proof.                    
		\end{proof}

\section{Bounds for the arrival velocity when connecting two given heights in a given time}\label{sec3}
The goal of this section is to prove Proposition \ref{prop2}, which  ensures the boundedness of the set of arrival velocities of solutions travelling between two given heights in a given time lapse. Throughout this section we assume {\bf [D$_{1}$]}.  

\medbreak

As a preliminary remark we point out that in the case of rectilinear solutions this statement was already proved in Corollary \ref{cor42}. Therefore, it only remains to check Proposition \ref{prop2} in the nonrectilinear situation. Remembering Proposition \ref{prop11}, nonrectilinear solutions are globally defined in the past.

\medbreak

We may divide this task into smaller steps using polar coordinates $x=re^{i\theta}$. Since $$|\dot x(0)|=\sqrt{\dot r(0)^2+r(0)^2\dot\theta(0)^2}=\sqrt{\dot r(0)^2+c(0)^2/r_B^2}\,,$$ we see that the nonrectilinear case of Proposition \ref{prop2} can be equivalently reformulated as follows: 

\medskip

\noindent
{\bf Proposition \ref{prop2}$^*$.}\ {\em For any $r_A,r_B,T>0$ there exists some $M>0$ such that whenever $x=x(t)$ is a nonrectilinear solution of \eqref{eu1}-\eqref{eu12} then
	$${(i)}\quad\dot r(0)\geq-M,\qquad{(ii)}\quad\dot r(0)\leq M,\qquad{(iii)}\quad |c(0)|\leq M.$$}

The goal of the remaining of this section is to prove the three assertions of Proposition \ref{prop2}$^*$. As shown in Proposition \ref{prop11}, nonrectilinear solutions are globally defined on the past. We begin with a result that clearly implies Proposition \ref{prop2}$^*${\em (i)}. 
\begin{lemma}\label{lem8}
	{Let $r_B,T>0$ be given, and let $x_n=x_n(t)$, $t\in[-T,0]$, be a sequence of solutions of \eqref{eu1} with $$r_n(0)=r_B,\qquad  \dot r_n(0)\to-\infty.$$ Then, $\dot r_n(t)\to-\infty\ \text{as }n\to+\infty,\text{uniformly with respect to } t\in[-T,0]$.
	}
	\begin{proof}

	After possibly discarding a finite number of terms there is no loss of generality in assuming that $\dot r_n(0)<0$ for every $n\in\mathbb N$. We set
		\begin{equation*}
		\tau_n:=\min\big\{t\in[-T,0]:\dot r_n(s)\leq 0\ \forall s\in[t,0]\big\}
		\end{equation*}
		and consider the functions $p_n:[\tau_n,0]\to\mathbb R$ defined as in \eqref{pp} for $x=x_n$. Recalling \eqref{eq101}, for each $n\in\mathbb N$ one has 
		\begin{equation*}
		\frac{d}{dt}\big(p_n(t)\dot r_n\big)\geq-\frac{p_n(t)}{r_n^2}\geq-\frac{1}{r_B^2},\qquad\tau_n\leq t\leq0,
		\end{equation*}
		and integration gives
		$$\dot r_n(t)\leq p_n(t)\dot r_n(t)\leq \dot r_n(0)+\frac{T}{r_B^2},\qquad\tau_n\leq t\leq 0.$$
		In particular, for $n$ big enough $\dot r_n(\tau_n)<0$ and we deduce that $\tau_n=-T$. Thus, $\dot r_n(t)\to-\infty$, uniformly with respect to $t\in[-T,0]$. It concludes the proof.
	\end{proof}
\end{lemma}

Our next task will consist in establishing the upper bounds announced in assertion {\em (ii)} of Proposition \ref{prop2}$^*$. This fact will arise from Lemma \ref{lem8} and statement {\em (iii)} of the following 
\begin{lemma}\label{lem9}
	{Let $x_n=x_n(t)$ be a sequence of nonrectilinear solutions of \eqref{eu1} satisfying, for some $r_*>0$,
$$r_n(0)=r_*,\qquad\dot r_n(0)\to+\infty.$$
Then, for $n$ big enough the following hold:
\begin{enumerate}
			\item[(i)] There exists some $\mathfrak t_n<0$ such that $\dot r_n(\mathfrak t_n)=0$. Moreover, $\lim_{n\to+\infty}\mathfrak t_n=0$. 
		\item[(ii)] $\dot r_n(t)>0\text{ for all }t\in]\mathfrak t_n,0[$, and $\dot r_n(t)<0\text{ for all }t<\mathfrak t_n\,.$
				
\item[(iii)] There exists some $\tau_n<\mathfrak t_n$ such that $r_n(\tau_n)=r_*$. Moreover, $\lim_{n\to+\infty}\tau_n=0$ and $\lim_{n\to+\infty}\dot r_n(\tau_n)=-\infty$.	
\end{enumerate}
}\end{lemma}

In the proof of Lemma \ref{lem9} we shall need the following result. Here, the potential and total energies $v=v(t)$, $h=h(t)$ are defined as in \eqref{euu}.
\begin{lemma}\label{lem4}
	{Let $x:[a,b]\to\mathbb R²\backslash\{0\}$ be a nonrectilinear solution of \eqref{eu1} satisfying
		$$\dot r(t)\geq 0\ \forall t\in[a,b],\qquad v(b)\geq 0\,.$$
		Then $\dot v\leq 0$ on $[a,b]$.
	}
	\begin{proof}A direct computation gives
		\begin{multline*}
		\dot v(t)=\dot r(t)\left(\frac{1}{r(t)^2}-\frac{c(t)^2}{r(t)^3}\right)+\frac{c(t)\dot c(t)}{r(t)^2}\leq\dot r(t)\left(\frac{1}{r(t)^2}-\frac{c(t)^2}{r(t)^3}\right)\leq\frac{\dot r(t)}{r(t)}\left(\frac{1}{r(t)}-\frac{c(b)^2}{r(t)^2}\right)\leq\\
		\leq-\frac{\dot r(t)}{r(t)}\left(\frac{c(b)^2}{2r(t)^2}-\frac{1}{r(t)}\right)
		\end{multline*}
		The function $x>0\mapsto c(b)^2/(2x^2)-1/x$ vanishes only at $x=c(b)^{-2}/2$, is positive for smaller $x$ and negative for bigger $x$. By assumption it is nonnegative at $x=r(b)$ and hence 
		$$0<r(t)\leq r(b)\leq\frac{ 1}{2c(b)^2},\qquad \frac{c(b)^2}{2´r(t)^2}-\frac{1}{r(t)}\geq 0,\qquad \text{ for any } t\in]a,b[,$$   	
		thus concluding  the proof.
	\end{proof}
\end{lemma}

\begin{proof}[Proof of Lemma \ref{lem9}] 
{\em (i)}: We use a contradiction argument: should this statement fail to hold, after possibly passing to a subsequence it would be possible to find some $\epsilon>0$ such that $\dot r_n(t)>0$ for every $t\in[-\epsilon,0]$. Recalling Lemma \ref{lem4} we distinguish three possibilities: 
\begin{enumerate}
	\item[{\em (a)}] $v_n(0)\to+\infty$. By Lemma \ref{lem4}, $v_n(t)\to+\infty$ as $n\to+\infty$, uniformly with respect to $t\in[-\epsilon,0]$, and in particular, $v_n(t)>0$ for all $t\in[-\epsilon,0]$ and $n$ big enough. We define $p_n:[-\epsilon,0]\to\mathbb R$ as in \eqref{pp} for $x=x_n$ and observe that
	\begin{equation}\label{eu31}
	\frac{d}{dt}\Big(p_n(t)\dot r_n(t)\Big)=p_n(t)\Big[\frac{c_n(t)^2}{r_n(t)^3}-\frac{1}{r_n(t)^2}\Big]\geq\frac{p_n(t)}{r_n(t)}v_n(t)\geq\frac{e^{-D_*\epsilon}}{{r_*}}v_n(t)\to+\infty\text{ as }n\to+\infty,
	\end{equation}
	uniformly with respect to $t\in[-\epsilon,0]$. Since $p_n(-\epsilon)\dot r_n(-\epsilon)>0$ we see that $p_n(t)\dot r_n(t)\to+\infty$ as $n\to+\infty$, uniformly with respect to $t\in[-\epsilon/2,0]$, contradicting the fact that  $r_*=r_n(0)$ for every $n$.
		\item[{\em (b)}] $\{v_n(-\epsilon)\}$ bounded from above. By Lemma \ref{lem4}, the sequence $\{v_n(t)\}$ is bounded uniformly with respect to $t\in[-\epsilon,0]$. On the other hand, for each $n\in\mathbb N$ we have
		$$\frac{1}{2}\dot r_n(t)^2=h_n(t)-v_n(t)\geq h_n(0)-v_n(t)\geq\frac{1}{2}\dot r_n(0)^2-v_n(t)\,,\qquad t\in[-\epsilon,0],$$
		and we deduce that $\dot r_n(t)\to+\infty$ as $n\to+\infty$, uniformly with respect to $t\in[-\epsilon,0]$. As before, it implies that $r_*=r_n(0)\to+\infty$ as $n\to+\infty$, a contradiction.
		\item[{\em (c)}] $v_n(-\epsilon)\to+\infty$ but $\{v_n(0)\}$ is bounded from above. Then, for $n$ big enough $$v_n(0)<\frac{1}{2}\min\big(v_n(-\epsilon),h_n(0)\big)<v_n(-\epsilon),$$ and there exists some $s_n\in]-\epsilon,0[$ such that $v_n(s_n)=\min\big(v_n(-\epsilon),h_n(0)\big)/2$. Thus, $v_n(s_n)\to+\infty$ as $n\to+\infty$ and $v_n(s_n)>0$ for $n$ big enough. Lemma \ref{lem4} implies that 
		$$v_n(t)\geq v_n(s_n)\text{ if }-\epsilon\leq t\leq s_n;\qquad v_n(t)\leq v_n(s_n)\text{ if }s_n\leq t\leq 0.$$ 
		
		Now, arguing as in case {\em (a)},
		for $t\in[-\epsilon,s_n]$ we have
		$$\frac{d}{dt}\Big(p_n(t)\dot r_n(t)\Big)\geq\frac{p_n(t)}{r_n(t)}v_n(t)\geq\frac{e^{-D_*\epsilon}}{{r_*}}v_n(s_n)\to+\infty\text{ as }n\to+\infty,$$
\end{enumerate}
from where it follows that $s_n\to-\epsilon$ as $n\to+\infty$. On the other hand, repeating the argument of case {\em (b)}, for $t\in[s_n,0]$ one has
	$$\frac{1}{2}\dot r_n(t)^2=h_n(t)-v_n(t)\geq h_n(0)-v_n(s_n)\geq\frac{1}{2}h_n(0)\to+\infty\text{ as }n\to+\infty,$$
	and we deduce that $s_n\to 0$, a contradiction again. It proves {\em (i)}.

\medbreak

{\em (ii)}: Let the sequence $\{\mathfrak t_n\}$ be as given by {\em (i).} After possibly replacing $\mathfrak t_n$ with $$\mathfrak t_n^*:=\max\{t<0:\dot r_n(t)=0\},$$ there is no loss if generality in assuming that $\dot r_n(t)>0$ for all $t\in]t_n,0[$. Since each $h_n=h_n(t)$ is a decreasing function of $t$, we see that
$$v_n(\mathfrak t_n)=h_n(\mathfrak t_n)\geq h_n(0)\geq\frac{1}{2}\dot r_n(0)^2-\frac{1}{r_*}\to+\infty\qquad\text{ as }n\to+\infty.$$
Thus, $v_n(\mathfrak t_n)>0$ for $n$ big enough, and we see that
$$\ddot r_n(\mathfrak t_n)=-D(x_n(\mathfrak t_n))\dot r_n(\mathfrak t_n)+\frac{c_n(\mathfrak t_n)^2}{r_n(\mathfrak t_n)^3}-\frac{1}{r_n(\mathfrak t_n)^2}\geq\frac{v_n(\mathfrak t_n)}{r_n(\mathfrak t_n)}>0,$$
and we deduce that $\dot r_n(t)<0$ for $0<\mathfrak t_n-t$ small. If there were some $\mathfrak s_n<\mathfrak t_n$ such that
\begin{equation*}
\dot r_n(\mathfrak s_n)=0,\qquad \dot r_n(t)<0\text{ for every }t\in]\mathfrak s_n,\mathfrak t_n[,
\end{equation*}
then Lemma \ref{lem4} would imply that $v_n(\mathfrak s_n)\geq v_n(\mathfrak t_n)>0$, and arguing as above, $\ddot r_n(\mathfrak s_n)>0$. This is a contradiction and concludes the proof.

\medbreak

{\em (iii)} Notice that $r_n(\mathfrak t_n)<r_*$. We set
$$\tau_n:=\inf\{t\in]-\infty,\mathfrak t_n[:r_n(t)<r_*\}\,,\qquad\mathfrak r_n:=\lim_{t\downarrow\tau_n}r_n(t)\leq r_*\,,$$
and write $$\dot r_n(t):=\begin{cases}
-\varphi_n(r_n(t))&\text{if }\tau_n<t<\mathfrak t_n\,,\\
\psi_n(r_n(t))&\text{if }\mathfrak t_n<t<0\,.
\end{cases}$$ The functions $\varphi_n:]r_n(\mathfrak t_n),\mathfrak r_n[\to\mathbb R$ and $\psi_n:]r_n(\mathfrak t_n),r_*[\to\mathbb R$ are continuous, and since all energy functions $h_n=h_n(t)$ are decreasing we see that $0<\psi_n(r)\leq\varphi_n(r)$ for every $r\in]r_n(\mathfrak t_n),\mathfrak r_n[$. Thus,
$$\mathfrak t_n-\tau_n=\int_{\tau_n}^{\mathfrak t_n}dt=-\int_{\tau_n}^{\mathfrak t_n}\frac{\dot r_n(t)}{\varphi_n(r_n(t))}dt=\int_{r_n(\mathfrak t_n)}^{\mathfrak r_n}\frac{1}{\varphi_n(r)}dr\leq\int_{r_n(\mathfrak t_n)}^{r_*}\frac{1}{\psi_n(r)}dr=-\mathfrak t_n,$$
i.e., $2\mathfrak t_n\leq\tau_n<0$. In particular, $\tau_n>-\infty$ and we see that $\mathfrak r_n=r_n(\mathfrak t_n)=r_*$. The result follows.
\end{proof}

We conclude this section by checking assertion {\em (iii)} of Proposition  \ref{prop2}$^*$. Having already shown the previous statements {\em (i)-(ii)}, the remaining work is collected in the following

\begin{lemma}\label{lem44}{Let $T,r_A,r_B>0$ be given, and let $\{x_n\}$ be a sequence of solutions of \eqref{eu1}-\eqref{eu12} such that $\{\dot r_n(0)\}$ is bounded. Then $\{c_n(0)\}$ is bounded.}
		\begin{proof}We distinguish two cases:

			\medbreak
			
			Assume firstly that $\max_{-T/2\leq t\leq 0}r_n(t)\to+\infty$. Since $r_n(0)=r_B$ for every $n$, there exists a sequence $\{t_n\}_n\subset[-T/2,0]$ with $r_n(t_n)\geq 1$ for every $n$ and $\dot r_n(t_n)\to-\infty$. Thus, Lemma \ref{lem8} states that $\max_{t\in[-T,t_n]}\dot r_n(t)\to-\infty$, implying that $r_A=r_n(-T)\to+\infty$, a contradiction.
			
			\medbreak
			
			The other possibility is that, after possibly passing to a subsequence, $\{r_n\}$ is uniformly bounded on $[-T/2,0]$. Using a contradiction argument we assume that  $|c_n(0)|\to+\infty$. Then $v_n(0)\to+\infty$ and Lemma \ref{lem4} states that $v_n(t)\to+\infty$, uniformly with respect to $t\in[-T/2,0]$. Arguing as in \eqref{eu31} we see that $(d/dt)(p_n(t)\dot r_n(t))\to+\infty$ uniformly with respect to $t\in[-T/2,0]$. With other words,
			$$\gamma_n:=\min_{t\in[-T/2,0]}\left[\frac{d}{dt}\big(p_n(t)\dot r_n(t)\big)\right]\to+\infty\text{ as }n\to+\infty.$$
			On the other hand, integration gives 
			$$p_n(t)\dot r_n(t)\leq\dot r_n(0)+\gamma_n t\leq |\dot r_n(0)|+\gamma_n t\,,\qquad t\in\left[-\frac{T}{2},0\right],$$
			and thus (by \eqref{sti}), $\dot r_n(t)\leq|\dot r_n(0)|e^{D_*T/2}+\gamma_nt$ for every $t\in[-T/2,0]$. Integrating again we find that $r_n(-T/2)\to+\infty$, contradicting our assumption that $\{r_n\}$ was uniformly bounded on $[-T/2,0]$. It proves the result.
		\end{proof}	 
\end{lemma}
\section{Collisions}\label{sec8}
In this section we continue the study, started in Section \ref{sec44}, of the rectilinear solutions of  the Kepler problem. Such solutions are governed by equation \eqref{1d}. An special emphasis will be put in studying the behaviour of solutions in connection with the singularity. As in Section \ref{sec44}, throughout this section we shall assume  that $\delta:]0,+\infty[\to\mathbb R$ is {\em nonnegative, bounded and continuously differentiable}.

\medbreak

Throughout this section let $r:]\alpha,\omega[\to\mathbb R$ be a solution of \eqref{1d}, assumed now maximal {\em both to the left and to the right}. Its associated energy is the function $h:]\alpha,\omega[\to\mathbb R$ defined by
	\begin{equation*}
	h(t):=\frac{1}{2}\dot r(t)^2-\frac{1}{r(t)}\,.
\end{equation*}
Notice that $h$ is decreasing. In fact, a direct computation gives
\begin{equation}\label{dhdt}
\dot h(t)=-\delta(r(t))\dot r(t)^2\leq 0.	
\end{equation}

\begin{lemma}\label{xlem51}
	{If $\alpha>-\infty$, then the following hold:
		\begin{enumerate}
			\item[(i)] $\displaystyle{\lim_{t\downarrow\alpha}r(t)=0}$.
			\item[(ii)] $\displaystyle{h(\alpha):=\lim_{t\downarrow\alpha}h(t)<+\infty}$.
			\item[(iii)] $\displaystyle{\lim_{t\downarrow\alpha}\frac{r(t)}{(t-\alpha)^{2/3}}=\sqrt[3]{\frac{9}{2}}}$\ . 
				\end{enumerate}}
\begin{proof}
{\em (i):} It follows from equation \eqref{1d} that at a critical point $t_0$, $\ddot r(t_0)<0$ and so $r$ attains a strict local maximum. In combination with Lemma \ref{lem31} it implies that $\lim_{t\downarrow\alpha}=0$, as claimed.

\medbreak

{\em (ii)} The previous argument actually gives some further information: there exists some $\alpha<t_0<\omega$ such that $\dot r(t)>0$ for all $t\in]\alpha,t_0[$. Using a contradiction argument we assume that $\lim_{t\downarrow\alpha}h(t)=+\infty$; then, after possibly replacing $t_0$ by a smaller number there is no loss of generality in further assuming that $h(t)>1$ for all $t\in]\alpha,t_0[$. Setting $r_0:=r(t_0)>0$ we see that there exists a $C^1$ function $\varphi:]0,r_0[\to\mathbb R$, $\varphi=\varphi(r)$, such that $h(t)=\varphi(r(t))$ for any $t\in]\alpha,t_0[$. Since
\begin{equation}\label{ue34}
\dot r(t)=\sqrt{2}\sqrt{\varphi(r(t))+\frac{1}{r(t)}},\qquad\alpha<t<t_0,
\end{equation}
differentiation and comparison with \eqref{dhdt} yields
$$0\geq\varphi'(r(t))=-\delta(r(t))\dot r(t)\geq-D_*\dot r(t)=-\sqrt{2}D_*\sqrt{\varphi(r(t))+\frac{1}{r(t)}},\qquad \alpha<t<t_0,$$
where $D_*\geq0$ is such that $\delta(r)\leq D_*$ for any $r>0$. Therefore,
$$\varphi'(r)\geq-\sqrt{2}D_*\sqrt{\varphi(r)+\frac{1}{r}}\geq-\sqrt{2}D_*\left(\sqrt{\varphi(r)}+\sqrt{\frac{1}{r}}\right)\geq-\sqrt{2}D_*\left(\varphi(r)+\frac{1}{\sqrt{r}}\right)\,, $$
for any $0<r<r_0$. Thus,
$$\frac{d}{dr}\left(e^{\sqrt{2}D_*r}\varphi(r)\right)\geq-\frac{\sqrt{2}D_*\,e^{\sqrt{2}D_*r}}{\sqrt{r}},\qquad 0<r<r_0,$$
which is not possible since $\lim_{r\downarrow 0}e^{\sqrt{2}D_*r}\varphi(r)=+\infty$ but the right hand side of the inequality is integrable on $]0,r_0[$. This contradiction concludes the proof.

\medbreak

{\em (iii)} By combining \eqref{ue34} and {\em (ii)} we see that
$$\lim_{t\downarrow\alpha}\sqrt{r(t)}\dot r(t)=\sqrt{2},$$
so that, by L'Hopital rule,
$$\lim_{t\downarrow\alpha}\frac{r(t)^{3/2}}{t-\alpha}=\frac{3}{\sqrt{2}},$$
implying the statement.
\end{proof}
\end{lemma}
Asymptotics of type {\em (iii)} were already obtained by Sperling \cite{Spe} (see also \cite[pp. 152-153]{Ort}) for the forced Kepler problem; however Sperling's results do not apply here directly since the damping force $-\delta(r)\dot r$ may not be bounded near the collision. We also remark that the corresponding version of Lemma \ref{xlem51}  when $t\uparrow\omega$ still holds if $\omega<+\infty$. In fact, statements {\em (i)} and {\em (iii)} can be readily translated to this situation with the same proofs. The situation in case {\em (ii)} is different: the result is still true but it needs a new proof. We shall skip the details since this is not needed in this paper; nevertheless, we emphasize the following consequence of the proofs of statements {\em (i)} and {\em (iii)} of Lemma \ref{xlem51}:

\begin{corollary}\label{xcor58}
	{If $\omega<+\infty$ and $\displaystyle{h(\omega):=\lim_{t\uparrow\omega}h(t)>-\infty}$, then $\displaystyle{\lim_{t\uparrow\omega}\frac{r(t)}{(\omega-t)^{2/3}}=\sqrt[3]{\frac{9}{2}}}$\ .}
\end{corollary}

We close this section with a result that estimates the length of the maximal definition interval of a solution from the final value (assumed finite) of its energy:

\begin{lemma}\label{lem52}{Assume that $-\infty<\alpha<\omega<+·\infty$ and $h(\omega)>-\infty$.
 Then, $h(\omega)<0$ and $\displaystyle{\omega-\alpha>-\frac{1}{2h(\omega)}}$.}
\begin{proof}
Corollary \ref{xcor58} implies in particular that $\lim_{t\uparrow\omega}r(t)=0$. Together with Lemma \ref{xlem51}{\em (i)} we see that there exists some point $t_0\in]\alpha,\omega[$ such that $\dot r(t_0)=0$. Since $h$ is decreasing, $h(t_0)=-1/r(t_0)\geq h(\omega)$; thus, $h(\omega)<0$ and $r_0:=r(t_0)\geq-1/h(\omega)$.

\medskip

The proof of Lemma \ref{xlem51}{\em (i)} implies that $\dot r(t)>0$ for every $t\in]\alpha,t_0[$. It follows that there exists an unique point $t_1\in]\alpha,t_0[$ such that $r(t_1)=r_0/2$,   and we see that 
$$\ddot r(t)\geq-\frac{1}{r(t)^2}>-\frac{4}{r_0^2},\qquad t_1<t<t_0,$$
 and integration gives
$$\frac{r_0}{2}=r(t_1)=r_0-\int_{t_1}^{t_0}\dot r(s)ds=r_0+\int_{t_1}^{t_0}(s-t_1)\ddot r(s)ds>r_0-\frac{4}{r_0^2}\int_{t_1}^{t_0}(s-t_1)ds=r_0-\frac{2}{r_0}(t_0-t_1)^2,$$
implying that 
$$\omega-\alpha>t_0-t_1>\frac{r_0}{2}\geq-\frac{1}{2h(\omega)},$$
and thus concluding the proof. 
\end{proof}
\end{lemma}

\section{The Levi-Civita regularization for the rectilinear Kepler problem}\label{sec6}
In this section we still assume  that $\delta:]0,+\infty[\to\mathbb R$ is nonnegative, bounded and continuously differentiable. In addition, mimicking  {\bf [D$_{2}$]} we introduce the hypothesis that $\lim_{r\downarrow 0}\sqrt{r}\,\delta'(r)=0$.
Let $r:]\alpha,0]\to]0,+\infty[$ be a solution of \eqref{1d} and set $$s(t):=\int_0^t\frac{d\tau}{r(\tau)},\quad\alpha<t\leq 0,\qquad\qquad A:=\lim_{t\downarrow\alpha}s(t)<0.$$ Then, the pair of functions $(u,E):]A,0]\to]0,+\infty[\times\mathbb R$ defined by
$$u(s(t)):=\sqrt{r(t)},\quad E(s(t)):=h(t),\qquad\qquad\alpha<t\leq 0,$$
solves the system
\begin{equation}\label{LC}
\begin{cases}
u''+\delta(u^2)u^2u'=\frac{Eu}{2},\\
E'=-4\delta(u^2)(u')^2,
\end{cases}\tag{{\em LC}}
\end{equation}
on the invariant manifold
$$\mathcal M:=\Big\{(u,u',E)\in]0,+\infty[\times\mathbb R\times\mathbb R:2(u')^2-Eu^2=1\Big\}.$$
(We denote by primes the derivatives with respect to the independent variable $s$; derivatives with respect to $t$ are denoted by dots).

\medbreak

Conversely, given a solution $(u,E):]A,0]\to]0,+\infty[\times\mathbb R$ of (\ref{LC}\,) on $\mathcal M$ we set
$$t(s):=\int_0^s u(\sigma)^2d\sigma,\quad A<s\leq 0,\qquad\qquad\alpha:=\lim_{s\downarrow A}t(s)<0,$$
and we see that the function $r=r(t)$ defined on the time interval $]\alpha,0]$ by $$r(t(s)):=u(s)^2,\qquad A<s\leq 0,$$ is a solution of \eqref{1d} with energy $h:]\alpha,0]\to\mathbb R$ given by $h(t(s))=E(s)$ for any $s\in]A,0]$.

\medbreak

In this situation we shall say that $(u,E)$ is the Goursat transform of $r$, and $r$ is the Levi-Civita transform of $(u,E)$. The Goursat and Levi-Civita transforms define mutually-inverse, bijective correspondences from the set of solutions $r=r(t)>0$ of \eqref{1d} defined on $]\alpha,0]$ for some $-\infty\leq\alpha<0$, into the set of solutions $(u,E)=(u(s),E(s))\in]0,+\infty[\times\mathbb R$ of (\ref{LC}\,) on $\mathcal M$ defined on $]A,0]$ for some $-\infty\leq A<0$. 

\medbreak

Notice now that (\ref{LC}\,) is naturally defined for $(u,E)\in\mathbb R^2$ and does not require that $u$ be positive. In fact, the map $u\in\mathbb R\mapsto\delta(u^2)$ is continuously differentiable thanks to our requirement that $\delta'$ is bounded near the origin.   This is the motivation behind the following statement. Here $r:]\alpha,0]\to]0,+\infty[$ is an arbitrary solution of \eqref{1d} and $(u,E):]A,0]\to\mathbb R^2$ stands for its Goursat transform. 

\medbreak

{\em If $(u,E)$ is maximal to the left as a solution of {\em (\ref{LC}\,)} then $r$ is also maximal to the left as a solution of \eqref{1d}. The converse is not true in general. }
	
	\medbreak
	
	In order to support the last part of this assertion we shall prove the following:
\begin{lemma}\label{lem61}{If $\dot r(t)>0$ for every $t\in]\alpha,0]$, then, $A>-\infty$ and $(u,E)$ is not maximal to the left as a solution of ({\em \ref{LC}}).}
\begin{proof}
If the interval $]\alpha,0]$ is not maximal to the left	for the solution $r$ of \eqref{1d} then the result is clear. Thus, we may henceforth assume that $]\alpha,0]$ is maximal to the left. Notice that
	$$\ddot r(t)=-\delta(r(t))\dot r(t)-\frac{1}{r(t)^2}\leq 0,\qquad t\in]\alpha,0],$$
so that $r$ is concave on $]\alpha,0]$. Consequently, its graph stays below that of its tangent line at $t=0$, i.e.,
$$r(t)\leq r(0)+\dot r(0)t,\qquad t\in]\alpha,0],$$
 and since $r$ is positive on $]\alpha,0]$ we see that $\alpha>-\infty$. Lemma \ref{xlem51} then applies and states that $A>-\infty$, $\lim_{s\downarrow A}E(s)<+\infty$, and  $\lim_{s\downarrow A}u(s)=0$. Since $(u,E)$ stays in $\mathcal M$ we see that $\lim_{s\downarrow A}u'(s)=1/\sqrt{2}$, and thus, the solution $(u,E)$ can be extended to the left of $A$. It proves the result.
\end{proof}	
\end{lemma}
\begin{lemma}\label{lem62}{Let $(u,E):]A,0]\to\mathbb R^2$ be a solution of ({\em \ref{LC}}) on $\mathcal M$, maximal to the left. If $A>-\infty$ then there exists some $s_1\in]A,0]$ such that $u'(s)u(s)<0$ for all $s\in]A,s_1]$.}
\begin{proof} We distinguish two cases depending on the sign of $E$.
	
	\medskip
	
	{\em Case I: There exists some $s_0\in]A,0]$ such that $E(s_0)\geq 0$.} The second equation of (\ref{LC}\,) implies that $E$ is decreasing, and we deduce that  $E(s)\geq 0$  for all $s\in]A,s_0]$. Since our solution stays on $\mathcal M$ we see that $u'(s)\not=0$ for all $s\in]A,s_0]$, and we deduce that there exists some $s_1\in]A,s_0]$ such that $u(s)\not=0$ for all $s\in]A,s_1]$.
	
	\medskip
	
	If $u'(s)u(s)<0$ for all $s\in]A,s_1]$ we are done; thus, let us assume that $u'(s)u(s)>0$ for all $s\in]A,s_1]$. By introducing a translation in the time variable $s$ there is no loss of generality in assuming that $s_1=0$, and after possibly replacing $u$ by $-u$ we may assume that $u(s),u'(s)>0$ for all $s\in]A,0]$. The Levi-Civita transform $r:]\alpha,0]\to]0,+\infty[$ of $(u,E)$ then satisfies that $\dot r(t)>0$ for all $t\in]\alpha,0]$, and Lemma \ref{lem61} implies that $(u,E)$ is not maximal, a contradiction. It proves the result in this case.
	
	\medskip
	
	{\em Case II: $E(s)<0$ for all $s\in]A,0]$.} Since $E$ is decreasing on $]A,0]$ it implies that $E$ is bounded on $]A,0]$. Since our solution stays in $\mathcal M$ we see that $2u'(s)^2\leq 1$ for all $s\in]A,0]$, and thus, $u'$ is also bounded on $]A,0]$. Finally, integration implies that $u$ is again bounded on $]A,0]$. The usual prolongation theory for ordinary differential equations implies that $(u,E)$ is not maximal, a contradiction. It concludes the proof.	\end{proof}
\end{lemma}
\begin{lemma}\label{lem63}
{Let $(u,E):]A,0]\to\mathbb R^2$ be a solution of (LC) on the invariant manifold $\mathcal M$, maximal to the left. Then $\int_{A}^0u(s)^2ds=+\infty$.}
\begin{proof} Let the strictly increasing function $t:]A,0]\to\mathbb R$ be defined by $t(s):=\int_0^s u(\sigma)^2d\sigma$, assume, by a contradiction argument, that $\alpha:=\lim_{s\downarrow A}t(s)>-\infty$, and let the function $r:]\alpha,0[\to\mathbb R$ be defined by $r(t(s)):=u(s)^2$. We shall distinguish three cases and find a contradiction in each of them.
	
	\medskip

{\em (i)} 	$A>-\infty$. Then, Lemma \ref{lem62} states the existence of some $s_1\in]A,0]$ such that $u'(s)u(s)<0$ for every $s\in]A,s_1]$. After a translation in the independent variable $s$ and possibly replacing $u$ by $-u$ we may assume that $s_1=0$ and $u'(s)<0<u(s)$ for all $s\in]A,0]$. Thus, $r:]\alpha,0]\to]0,+\infty[$ is the Levi-Civita transform of $(u,E)$, and in particular, it is a solution of \eqref{1d}, maximal to the left. It satisfies $\dot r(t)<0$ for every $t\in]\alpha,0]$, and Lemma \ref{xlem51}{\em (i)} implies that $\alpha=-\infty$. It concludes the proof in this case.

\medskip

{\em (ii)} $A=-\infty$ and $u$ has infinitely many zeroes. The definition of the manifold $\mathcal M$ implies that these zeroes are nondegenerate, and in particular isolated; thus, they make up an ordered sequence $\ldots<s_2<s_1<s_0\leq 0$. The sequence $t_i:=t(s_i)$ is strictly decreasing and convergent, and  for each $i\geq 1$ the restriction of $r$ to $]t_{i},t_{i-1}[$ is a maximal solution of \eqref{1d} satisfying $h(t_{i-1})>-\infty$. Then, Lemma \ref{lem52} states that the sequence of energies $h(t_{i}):=\lim_{t\downarrow t_{i}}h(t)$ satisfies $h(t_{i})\to-\infty$. On the other hand $h(t(s))=E(s)$, and the second equation of system $(LC)$ implies that $\{h(t_{i})\}$ is increasing, a contradiction.

\medskip

{\em (iii)}  $A=-\infty$ and the set $Z$ of zeroes of $u$ in $]-\infty,0]$ is finite. Write $$\{-T,0\}\cup t(Z)=\{-T=t_p<t_{p-1}<\ldots<t_0=0\},$$ and observe that the restriction of $r$ to each interval $]t_i,t_{i-1}[$ is a solution of \eqref{1d} satisfying $h(t_{i-1})>-\infty$. The combination of Lemma \ref{xlem51}{\em (iii)} and Corollary \ref{xcor58} then implies that $\int_{t_{i}}^{t_{i-1}}\frac{1}{r(t)}dt<+\infty$ for each $i$, and therefore $\int_{-T}^0\frac{1}{r(t)}dt<+\infty$, a contradiction. The proof is complete.
\end{proof}
\end{lemma}

\section{The Levi-Civita regularization for the planar Kepler problem}\label{sec7}
In this section we prove Proposition \ref{prop12}, which was key in our proof of Theorem \ref{th1}. Henceforth we assume both {\bf [D$_{1-2}$]}.

\medbreak 

The Levi-Civita regularization applies not only to the 1-dimensional Kepler problem \eqref{1d} but also to the more general planar problem \eqref{eu1}. The well-known procedure goes as follows. Let $x:]\alpha,0]\to\mathbb R^2\backslash\{0\}\equiv\mathbb C\backslash\{0\}$ be a given solution of \eqref{eu1}, and let $w_0\in\mathbb C$ be such that $w_0^2=x_0$. The transformation from $\mathbb C\backslash\{0\}$ to itself given by $z\mapsto z^2$ is a covering map, and thus, there exists a (unique) continuous lifting $z:]\alpha,0]\to\mathbb C\backslash\{0\}$ with $z(0)=w_0$ and $z(t)^2=x(t)$ for every $\alpha<t\leq 0$. Let $h,s:]\alpha,0]\to\mathbb R$ (energy and new time) be defined by
$$h(t):=\frac{1}{2}\,|\dot x(t)|^2-\frac{1}{|x(t)|},\qquad s(t):=\int_0^t\frac{1}{|x(\tau)|}\,d\tau\,.$$ Then, setting $A:=\lim_{t\downarrow\alpha}s(t)$ we see that the pair of functions $(w,E):]A,0]\to(\mathbb C\backslash\{0\})\times\mathbb R$ defined by
$$w(s(t)):=z(t),\qquad E(s(t)):=h(t),$$
solves the system
\begin{equation*}
	(\widehat{LC})\begin{cases}
		\displaystyle{w''+D(w^2)|w|^2w'=\frac{Ew}{2}}\\
		E'=-4D(w^2)|w'|^2
	\end{cases},\qquad w\in\mathbb C,\ E\in\mathbb R,
\end{equation*}
on the invariant manifold
$$\widehat{\mathcal M}:=\Big\{(w,w',E)\in(\mathbb C\backslash\{0\})\times\mathbb C\times\mathbb R:2|w'|^2-E|w|^2=1\Big\}.$$
In addition, $w(0)=w_0$. Notice that picking $-w_0$ instead of $w_0$ leads to the pair $(-w,E)$ in the place of $(w,E)$.

\medbreak

Conversely, given a solution $(w,E):]A,0]\to(\mathbb C\backslash\{0\})\times\mathbb R$ of $(\widehat{LC})$ on $\widehat{\mathcal M}$, letting $t:]A,0]\to\mathbb R$ be defined by 
$$t(s):=\int_0^s|w(\sigma)|^2d\sigma,\qquad A<s\leq 0,$$ 
and setting $\alpha:=\lim_{s\downarrow A}t(s)$, we see that the function $x=x(t)$ defined on $]\alpha,0]$ by $x(t(s)):=w(s)^2$ is a solution of \eqref{eu1} with energy $$h:]\alpha,0]\to\mathbb R,\qquad h(t(s))=E(s)\text{ for any }A<s\leq 0,$$
and moreover, $x(0)=w_0^2$ for $w_0:=w(0)$.

\medbreak

Under these circumstances we shall say that $(w,E)$ is the Goursat transform of $(x,w_0)$ and $(x,w_0)$ is the Levi-Civita transform of $(w,E)$. With this in mind, the Goursat and Levi-Civita transforms define mutually-inverse, bijective correspondences between the set of couples $(x,w_0)$, where $x=x(t)$ is a solution of \eqref{eu1} defined on $]\alpha,0]$ for some $-\infty\leq\alpha<0$ and $w_0$ is a choice of the square-root of $x(0)$, and the set of pairs $(\pm w,E)$, where $(w,E)=(w(s),E(s))\in(\mathbb C\backslash\{0\})\times\mathbb R$ is a solution  of (\ref{LC}\,) on $\widehat{\mathcal M}$ defined on $]A,0]$ for some $-\infty\leq A<0$. 

\medbreak

Observe also that the initial conditions of $w,E$ and $x$ at $t=0$ are related as follows: $w(0)$ is a square root of $x(0)$,  ${\displaystyle\dot w(0)=\frac{|x(0)|\,\dot x(0)}{2w(0)}}$, and $\displaystyle{E(0)=\frac{|\dot x(0)|^2}{2}-\frac{1}{|x(0)|}}$. This fact will be used later.

\medbreak

At this moment one might ask about the connections between the Goursat or Levi-Civita transforms (as defined in the previous section in the one dimensional situation), and the newly-defined notions of Goursat/Levi-Civita transforms for the planar problem. With this goal let
$x(t)=r(t)v_0,\ t\in]\alpha,0]$, be a rectilinear solution of \eqref{eu1}. We assume that $|v_0|=1$ and $r$ is positive, so that $r$ must be a solution of \eqref{1d} for $\delta(r):=D(rv_0)$. Then, letting $(u,E)$ be the Goursat transform of $r$, the Goursat transform $(\pm w,E)$ of $x$ is given by $w(s)=u(s)v_1$, where $v_1\in\mathbb C$ is one of the square roots of $v_0$. Similarly, if $(w,E):]-A,0]\to(\mathbb C\backslash\{0\})\times\mathbb R$ is a solution of $(\widehat{LC})$ on $\widehat{\mathcal M}$, and this solution is rectilinear in the sense that $w(s)=u(s)v_1$ for some function $u:]A,0]\to]0,+\infty[$ and some $v_1\in\mathbb C$ with $|v_1|=1$, then $(u,E)$ must be a solution of (\ref{LC}\,) in $\mathcal M$ for $\delta(u):=D(uv_1^2)$, and the Levi-Civita  transform $x$ of $(\pm w,E)$ is given by $x(t)=r(t)v_1^2$, where $r=r(t)$ is the Levi-Civita transform of $(u,E)$.

\medbreak

On the other hand, the initial conditions in the rectilinear case are related as follows: if $x(0)=x_0$ and $\dot x(0)=\lambda x_0$ for some $x_0\in\mathbb C\backslash\{0\}$ and some $\lambda\in\mathbb R$, then $u(0)=|x_0|$, $u'(0)=|x_0|^{3/2}\lambda/2$, and $E(0)=\lambda^2|x_0|^2/2-1/|x_0|$.

\medbreak

We finally notice that, as a consequence of assumption {\bf [D$_2$]}, the map $\mathbb C\to\mathbb R$, $w\mapsto D(w^2)$ is continuously differentiable in the real sense (i.e., regarded as a map from $\mathbb R^2$ to $\mathbb R$). Thus, we can see $(\widehat{LC})$ as a $C^1$, autonomous system on $\mathbb C\times\mathbb R\equiv\mathbb R^2\times\mathbb R$.  

\begin{proof}[Proof of Proposition \ref{prop12}] It suffices to check that given a converging sequence $$\{(x_0^{(n)},\dot x_0^{(n)})\}_n\to(x_0^{*},\dot x_0^{*}),$$ with $(x_0^{(n)},\dot x_0^{(n)})\in\Omega$ for every $n\in\mathbb N$, $x_0^*\not=0$, $\dot x_0^*=\lambda x_0^*$ for some $\lambda\in\mathbb R$, and $(x_0^{*},\dot x_0^{*})\in\partial\Omega$, then $\big\{\mathcal X(x_0^{(n)},\dot x_0^{(n)})\big\}$ is convergent. 
	
	\medbreak

	With this goal, for each natural index $n$ we denote by $x_n:]\alpha_n,0]\to\mathbb R^2\backslash\{0\}\equiv\mathbb C\backslash\{0\}$ the solution of \eqref{eu1} with $x_n(0)=x_0^{(n)}$ and $\dot x^{(n)}(0)=\dot x_0^{(n)}$. The definition interval $]\alpha_{n},0]$ is chosen maximal to the left. The points  $(x_0^{(n)},\dot x_0^{(n)})$ being in $\Omega$ we see that $\alpha_n<-T$ for every $n\in\mathbb N$.
	
	\medbreak

	Similarly, we call $x_*:]\alpha_*,0]\to\mathbb C\backslash\{0\}$ the  solution (maximal to the left) of \eqref{eu1} with $x_*(0)=x_0^{*}$ and $\dot x_n(0)=\dot x_0^*=\lambda x_0^{*}$. The fact that $(x_0^{*},\lambda x_0^{*})\not\in\Omega$ can be equivalently rewritten as $\alpha_*\geq-T$.
	
	\medbreak

	For each $n\in\mathbb N$ we choose some $w_0^{(n)}\in\mathbb C$ with $(w_0^{(n)})^2=x_0^{(n)}$, and pick some $w_0^*\in\mathbb C$ such that $(w_0^*)^2=x_0^*$. These choices are to be made in such a way that $$w_0^{(n)}\to w_0^*\text{ as }n\to+\infty.$$ 
	
	For each natural index $n\in\mathbb N$ we also denote by $(w_n,E_n):]A_n,0]\to\mathbb C\times\mathbb R$ the Goursat transform of $(x_n,w_0^{(n)})$. It is a solution of $(\widehat{LC})$, and therefore, it can be extended to a possibly bigger interval $]\hat A_n,0]$, maximal to the left. Similarly, the Goursat transform $(w_*,E_*):]A_*,0]\to\mathbb C\times\mathbb R$ of $(x_*,w_0^*)$ will be extended to some greater interval $]\hat A_*,0]$,  maximal to the left. Notice that
	$$w_*(s)=u_*(s)\frac{w_0^*}{|w_0^*|},\qquad s\in]\hat A_*,0],$$
	where $(u_*,E_*):]\hat A_*,0]\to\mathbb R^2$ is a maximal solution of  (\ref{LC}\,) for $\delta(u):=D(ux_0^*/|x_0^*|)$. Observe also that it stays on the invariant manifold $\mathcal M$.

	\medbreak
	
	Lemma \ref{lem63} states that $\int_{\hat A_*}^0|w_*(s)|^2ds=\int_{\hat A_*}^0u_*(s)^2ds=+\infty$. Therefore, there exists  a unique point $S_*\in]A_*,0[$ such that $\int_{S_*}^0u_*(s)^2ds=T$. The remaining of this proof is devoted to show that
	\begin{equation*}
	\mathcal X(x_0^{(n)},\dot x_0^{(n)})\to u_*(S_*)^2\frac{x_0^*}{|x_0^*|} \qquad\text{ as }n\to+\infty.
	\end{equation*}
	With this goal we observe that 
	\begin{equation*}
	w_n(0)=w_0^{(n)}\to w_0^*=w_*(0),\quad \dot w_n(0)=\frac{\left|x_0^{(n)}\right|\dot x_0^{(n)}}{2w_0^{(n)}}\to \frac{\left|x_0^{*}\right|\dot x_0^{*}}{2w_0^{*}},\qquad\text{as }n\to+\infty,
	\end{equation*}
	and similarly, $E_n(0)\to E_*(0)$. Therefore, continuous dependence implies that $w_n(s)\to w_*(s)$ uniformly with respect to $s\in[s_*-1,0]$. In particular, for $n$ big enough we see that
	$\int_{S_*-1}^0|w_n(s)|^2ds>T$ and there exists some $S_n\in]S_*-1,0[$ such that $\int_{S_n}^0|w_n(s)|^2ds=T$. After possibly passing to a subsequence we may assume that the sequence $\{S_n\}\to S_{**}\in[S_*-1,0]$, and since $w_n\to w_*$ uniformly on $[S_*-1,0]$ we see that $T=\int_{S_n}^0|w_n(s)|^2ds\to\int_{S_{**}}^0|w_*(s)|^2ds$. Therefore, $\int_{S_{**}}^0|w_*(s)|^2ds=T$ and we deduce that $S_{**}=S_*$. Consequently, $\mathcal X(x_0^{(n)},\dot x_0^{(n)})=w_n(S_n)^2\to w_*(S_*)^2$, thus concluding the proof.\end{proof}

{\em Acknowledgements:} 
I thank A. Albouy for introducing me to the Lambert problem. I owe him many classical references, including \cite{Bop, Eli, Eul, Gau, Sim}.

\medbreak

I am indebted to R. Ortega for fruitful discussions leading to the present form of the paper and for pointing out several misprints in a previous version of the manuscript.

\end{document}